\documentclass[11pt,english]{smfart}
  
\usepackage[french,main=english]{babel}
\usepackage{stix2,enumerate,dsfont,a4wide}
  
\newcommand{\inv}{^{\raisebox{.2ex}{$\scriptscriptstyle-1$}}} 
  
\newcommand{\id}{\mathcal{I}\!\textit{dl}(S)}

\newcommand{\ids}{\mathcal{I}\!\textit{dl}_{\scriptscriptstyle{\mathrm{sub}}}(S)}

\newcommand{\idss}{\mathcal{I}\!\textit{dl}_{\scriptscriptstyle{\mathrm{sub}}}(S')}

\newcommand{\cs}{\mathcal{C}_{\scriptscriptstyle{\mathrm{sub}}}}

\usepackage[dvipsnames]{xcolor}   
\usepackage{xr-hyper}
\definecolor{brightmaroon}{rgb}{0.76, 0.13, 0.28}
\usepackage[linktocpage=true,colorlinks=true,hyperindex,citecolor=Royal Blue,linkcolor=Royal Blue]{hyperref}   

\newtheorem{theorem}{Theorem}[section]
\newtheorem{proposition}[theorem]{Proposition}
\newtheorem{lemma}[theorem]{Lemma}
\newtheorem{corollary}[theorem]{Corollary}
\theoremstyle{definition}

\numberwithin{equation}{section}

\begin{document}

\author{Amartya Goswami}

\address{
[1] Department of Mathematics and Applied Mathematics, University of Johannesburg, P.O. Box 524, Auckland Park 2006, South Africa. 
[2] National Institute for Theoretical and Computational Sciences (NITheCS), South Africa.}

\email{agoswami@uj.ac.za}

\title{Subtractive spaces of semirings}

\date{}

%\thanks{This project has received funding from the .} 


% Semirings


\begin{abstract}
Using the closure operator that defines a subtractive ideal of a semiring $S$, in this note we introduce a topology on the set of all ideals of $S$ induced by that operator. We show that the corresponding subtractive space is $T_0$ and every nonempty irreducible closed set has a unique generic point, whereas the restricted subspace of subtractive ideals is $T_1$. Using a semiring homomorphism, we obtain a continuous map between the corresponding subtractive spaces.
\end{abstract}
\maketitle

\section{Introduction and Preliminaries}

Since the introduction of semirings in \cite{V34}, it is natural to compare and extend results from rings to semirings. One may think that semirings can always be extended to rings, but  \cite{V39} gives examples of semirings that can not be embedded in rings. Furthermore, the lack of `subtraction' in a semiring makes the behaviour of ideals substantially different from that of rings. To minimize this gap, the notion of a `$k$-ideal' (also called a subtractive ideal in \cite{G99}) has been introduced in \cite{H58}. There is a natural closure operator that defines a subtractive ideal. The aim of this note is to introduce a topology on the set of all ideals of a semiring induced by this closure operator. We study a few topological properties of these spaces.
  
A (commutative) \emph{semiring} is a system $(S ,+,0,\cdot, 1)$ such that $(S ,+,0)$ is a commutative monoid, $(S , \cdot,1)$ is a commutative monoid, $0\cdot x=0=x\cdot 0$ for all $x\in S ,$ and $\cdot$ distributes over $+$. We shall write $x y$ for $x\cdot y.$
A \emph{semiring homomorphism} $\phi\colon S \to S'$ is a map such that $\phi(x+y)=\phi(x)+\phi(y),$ $\phi(xy)=\phi(x)\phi(y),$ and $\phi(1)=1$ for all $x,$ $y\in S.$ 
An \emph{ideal} $I $ of a semiring $S$ is an additive submonoid of $S$ such that 
$rx\in I $
for all $x\in I $ and $r\in S .$ An ideal $I$ is called \emph{proper} if $I\neq S.$ We also use the symbol $0$ to denote the zero ideal of $S$.

Recall from \cite{G99} that a \emph{subtractive ideal}  $I$ of $S$ is an ideal of $R$ such that $x,$ $x+y\in I$ imply $y\in I$. Equivalently, an
ideal $I$ of $S$ is called a $k$-ideal if $x + y \in I$ implies either $x,$ $y \in I$ or $x,$ $y \notin
I.$ Surely, the zero ideal is subtractive and is contained in every $k$-ideal of $S$. We denote the set of all ideals and all subtractive ideals of $S$ by $\id$ and $\ids$ respectively. 
The notion of a subtractive ideal can also be characterized in terms of a closure operator endowed on $\id$. Suppose $I\in \id$. A \emph{subtractive closure} (also known as $k$-closure, see \cite[Lemma 2.2]{SA92}) operator $\cs$ is defined by
\begin{equation}
\label{clkdef}
\cs(I)=\{r\in S\mid r+x\in I\;\text{for some}\; x\in I\}.
\end{equation} 

\begin{lemma}\label{lclk}
Let $I$, $\{I_{\lambda}\}_{\lambda \in \Lambda}$, and $J$ be ideals of a semiring $S$. Then  $\cs$ has the following properties.
\begin{enumerate}[\upshape (1)]
	
\item\label{iclk} $I\subseteq \cs(I).$

\item \label{zgz}
$\cs(0)=0.$

\item\label{rclk} $\cs(S)=S.$

\item\label{clcl} $\cs(\cs(I))=\cs(I).$

\item\label{ijcl} $I\subseteq J$ implies $\cs(I)\subseteq \cs(J).$

\item \label{clu}
$\cs(I\cup J)\supseteq \cs(I) \cup \cs(J).$

\item\label{arbin} $\cs\left( \bigcap_{\lambda\in \Lambda}I_{\lambda}\right)=\bigcap_{\lambda\in \Lambda} \cs (I_{\lambda}).$

\item\label{clksm} $\cs(I)$ is the smallest subtractive ideal containing $I$.

\item\label{altd} $I$  is a subtractive if and only if $I=\cs(I).$
\end{enumerate}
\end{lemma}

\begin{proof}
The proofs of (\ref{iclk})--(\ref{arbin}) are straightforward. For  (\ref{clksm}), see \cite[Proposition 3.1]{JRT22}, whereas for (\ref{altd}), we refer to \cite[Lemma 2.2]{SA92}.
\end{proof}

It is obvious from from Lemma \ref{lclk}(\ref{ijcl}) that $\cs(I)\subseteq \cs(\sqrt{I})$ for all $I\in \id.$ Note that we may use Lemma \ref{lclk}(\ref{altd}) as an alternative definition of a subtractive ideal of a semiring. From (\ref{clksm}),  it follows that a $\cs$ is indeed a map
\[\cs\colon \id\to \ids\]
defined by (\ref{clkdef}). Considering the inclusion map $\iota\colon \ids\to \id$, it is easy to see the following.

\begin{proposition}
The pair $(\cs, \iota)$ forms a Galois connection.
\end{proposition}

\begin{lemma}\label{psi}
If $I$ and $J$ are two subtractive ideals of a semiring $S$, then their product $IJ$ is also a subtractive ideal of $R$, and $IJ\subseteq I\cap J.$
\end{lemma}

\begin{proof}
Suppose $x,$ $x+y\in IJ$. Then $x=ij$ and $x+y=i'j'$ for some $i,$ $i'\in I$ and $j,$ $j'\in J.$ Since $I$ is an ideal, $ij,$ $i'j'\in I$, that is, $x,$ $x+y\in I$. Since $I$ is also a subtractive ideal, this implies $y\in I.$ Similarly, we can show that $y\in J.$ Hence, $y\in IJ,$ and this proves that $IJ$ is a subtractive ideal. 

Let $r\in IJ$. Then there exists an $x\in IJ$ such that $r+x\in IJ.$ Since $I$ and $J$ are ideals of $S$, we definitely have $IJ\subseteq I\cap J$, which implies that $r+x,$ $x\in I \cap J$. Since $I$ and $J$ are also subtractive ideals, we must have $r\in I\cap J$, as required. 
\end{proof}

If $\{I _{\lambda}\}_{\lambda\in \Lambda}$ is a family of subtractive ideals, then their intersection $\bigcap_{\lambda\in \Lambda} I _{\lambda}$ is also a subtractive ideal. Note that the sum of two subtractive ideals of a semiring need not be a subtractive ideal. Recall from \cite[Example 6.19]{G99} that $2\mathds{N}$ and $3\mathds{N}$ are subtractive ideals of the semiring $\mathds{N}$, however $2\mathds{N}+3\mathds{N}=\mathds{N}\setminus \{1\}$ is not a subtractive ideal of $\mathds{N}$, however it is so in a lattice ordered semiring (\textit{cf}.
\cite[Corollary 21.22]{G99}). 
The lattice of all ideals of a ring is modular, whereas  the same is not true for a semiring. Nevertheless, we have the following result that announced in \cite{H58}. For a proof, see \cite[Proposition 6.38]{G99}.

\begin{proposition}
Let $S$ be a semiring. Then $\ids$ is a modular lattice.
\end{proposition}

\section{Subtractive spaces} 

From Lemma \ref{lclk}(\ref{iclk})--(\ref{ijcl}), we observe that a closure operator $\cs$ satisfies the axioms of an algebraic closure operators. However, it is not true in general that $\cs$ is closed under finite unions, and hence, it is not a Kuratowski closure operator. Considering the subsets $\{\cs(I)\}_{I\in \id}$ of $S$ as subbasic closed sets, nevertheless, induce a topology on $\id$,  which we call a \emph{subtractive topology} and denote by $\tau_s$. For a semiring $S$, the set $\id$ endowed with a subtractive topology is called a \emph{subtractive space}, and instead of $(\id, \tau_s)$, we denote the space also by $\id$.

\begin{lemma}
The subbasic closed sets of a subtractive space $\id$ are the subtractive ideals of $S$.
\end{lemma}

\begin{proof}
The proof follows from  (\ref{clksm}) and (\ref{altd})) of Lemma \ref{lclk}.
\end{proof}

If $I,$ $I'\in \id$ and $I\neq I'$, then it is easy to see that $\cs(I) \neq \cs(I')$, and hence we have 

\begin{lemma}\label{t0}
Every subtractive space is $T_{{\scriptscriptstyle 0}}.$
\end{lemma}

Recall that a nonempty closed subset $D$ of a topological space is \emph{irreducible} if $D\neq D_{\scriptscriptstyle 1}\cup D_{\scriptscriptstyle 2}$ for any two proper closed subsets  $D_{\scriptscriptstyle 1}$ and $D_{\scriptscriptstyle 2}$ of $D$. A point $x$ in a closed subset $D$ is called a \emph{generic point} of $D$ if $D = \overline{\{x\}}.$  

\begin{proposition}\label{scir}
Every nonempty subbasic closed set of a subtractive space is irreducible.
\end{proposition}

\begin{proof}
We show that $\cs(I)=\overline{\{I\}}$ for all $I\in \id.$ Since $\overline{\{I\}}$ is the smallest closed set containg $I$, it follows from Lemma \ref{lclk}(\ref{iclk}) that $\cs(I)\supseteq\overline{\{I\}}$. To have the other inclusion, first consider the trivial case of $\overline{\{I\}}=\id.$ For this we have \[\id=\overline{\{I\}}\subseteq \cs(I)\subseteq \id,\] and hence $\cs(I)\subseteq \overline{\{I\}}.$ Now suppose \[\overline{\{I\}}=\bigcap_{\lambda \in \Lambda}\left(\bigcup_{i=1}^{n_{\lambda}}\cs (I_{i\lambda})  \right).\]
This means that $I\subseteq \cs (I_{i\lambda})$ for some $i$ and each $\lambda \in \Lambda$. But that implies  
\[\cs (I) \subseteq \cs\left(  \cs (I_{\lambda i})\right)=\cs(I_{\lambda i})\subseteq \bigcap_{\lambda \in \Lambda}\left(\bigcup_{i=1}^{n_{\lambda}}\cs (I_{i\lambda})  \right),\] and hence we have the desired inclusion.
\end{proof}

\begin{corollary}
$\ids$ is the largest $T_{{\scriptscriptstyle 1}}$-subspace of a subtractive space  $\id$.
\end{corollary}

\begin{proof}
If $I\in \ids,$ then by Proposition \ref{scir}, we have $I=\cs(I) = \overline{\{I\}}.$
\end{proof}

\begin{theorem}
Every nonempty irreducible closed subset of a subtractive space has a unique generic point.
\end{theorem}

\begin{proof}
Suppose $D$ is a nonempty irreducible closed subset of a subtractive space $\id$. Then $D=\bigcap_{\lambda \in \Lambda} \mathcal{E}_{\lambda}$, where each $\mathcal{E}_{\lambda}$ is a finite union of subbasic closed sets of $\tau_s.$ Since $D$ is irreducible, for every $\lambda \in \Lambda$, there exists an $I_{\lambda}\in \id$ such that 
\[D\subseteq \cs(I_{\lambda})\subseteq \mathcal{E}_{\lambda},\]
and this implies
\[D=\bigcap_{\lambda \in \Lambda}\cs(I_{\lambda})=\cs\left( \bigcap_{\lambda\in \Lambda} I_{\lambda} \right)=\overline{\left\{\bigcap_{\lambda\in \Lambda}I_{\lambda}\right\}},\]
where, the last equality follows from Proposition \ref{scir}. This proves the existence of the generic point, whereas the uniqueness of it follows from Lemma \ref{t0}.
\end{proof}

Using a semiring homomorphism, our aim in this section is to construct a continuous map between the corresponding subtractive spaces and study some of the properties of these maps. The main difference compared to Zariski topology is that we have to use subbasic-closed-set  formulation to study these maps. 

\begin{lemma}\label{pssj}
If $\phi\colon S\to S'$ is a semiring homomorphism and $J\in \idss$. Then the following hold.
\begin{enumerate}[\upshape(1)]
	
\item\label{kers} $\phi\inv(J)$ is a subtractive ideal of $S$. In particular, $\mathrm{ker}\phi$ is a subtractive ideal of $S$.

\item\label{pjcs} $\phi\inv(J)=\cs(\phi\inv(J)).$
\end{enumerate}
\end{lemma}

\begin{proof}
For the first part of (\ref{kers}), it is well-known that $\phi\inv(J)\in \id.$ Suppose $x,$ $x+y\in \phi\inv(J).$ Then $\phi(x),$ $\phi(x+y)=\phi(x)+\phi(y)\in J$. Since $J\in \idss,$ we must have $\phi(y)\in J,$ and hence $y\in \phi\inv(J).$ For the second part of (\ref{kers}), let $x,$ $x+y\in \mathrm{ker}\phi.$ This implies $\phi(y)=\phi(x)+\phi(y)=\phi(x+y)=0,$ and hence $y\in \mathrm{ker}\phi.$ The proof of (\ref{pjcs}) follows immediately from (\ref{kers}).
\end{proof}

\begin{proposition}\label{conmap}
Suppose $\phi\colon S\to S'$ is a semiring homomorphism.
\begin{enumerate}[\upshape(1)]
		
\item \label{contxr} The map $\phi$ induces a continuous map $\phi_!\colon  \mathcal{I}\!\textit{dl}(S')\to \id$ defined by  $\phi_!(J)=\phi\inv(J)$, where $J\in\mathcal{I}\!\textit{dl}(S').$ 
		
\item \label{shcs} If $\phi$ is  surjective, then the subtractive spaces $\idss$ and  $\ids$ are homeomorphic.
\end{enumerate}
\end{proposition}

\begin{proof}      
To show (\ref{contxr}), let $\cs(I)$ be a   subbasic closed set of the subtractive space $\id$, and for us it is sufficient to show that $\phi(\cs(I))\subseteq \cs(\langle \phi(I)\rangle).$ Let $s'\in \phi(\cs(I)).$ This implies $\phi\inv(s')\in \cs(I),$ and hence
\[\phi\phi\inv(s')+\phi(i)\in \phi(I)\subseteq \langle \phi(I)\rangle,\]
for some $i\in I.$ From this we conclude that $s'\in \phi\phi\inv(s')\in \cs(\langle \phi(I)\rangle).$   

For (\ref{shcs}), it is easy to see that the map $\phi_!$ is injective. Since by hypothesis $\phi_!$ is surjective and by (\ref{contxr}), $\phi_!$ is continuous, what remains is to show that $\phi_!$ is closed. Notice that if $\cs(I)$ is a subbasic closed subset of  $\idss$, then by Lemma \ref{pssj}(\ref{pjcs}), $\phi_!(\cs(I))$ is also a subbasic closed set of $\idss.$ Now if $K$ is a closed subset of $\idss$, then there exists a collection  $\{I_{j\lambda}\mid \lambda \in \Lambda, 1\leqslant j\leqslant m_{\lambda} \}$ of subtractive ideals of $S'$ such that
\begin{align*}
\phi_!(K)&=\phi_!\left(\bigcap_{ \lambda \in \Lambda} \left(\bigcup_{ j \,= 1}^{ m_{\lambda}} \cs(I_{ j\lambda})\right)\right)\\&=\phi_!\left(\bigcap_{ \lambda \in \Lambda} \left(\bigcup_{ j \,= 1}^{ m_{\lambda}} I_{ j\lambda}\right)\right)\\&=\bigcap_{ \lambda\in \Lambda} \bigcup_{ j = 1}^{ m_{\lambda}} \phi\inv(I_{ j\lambda}),
\end{align*} 
a closed subset of  $\ids.$ 
\end{proof}

%\section*{Acknowledgement}


\begin{thebibliography}{1}
	
%\bibitem{AK13} A. Altman and S. Kleiman, A term of commutative algebra, \textit{Worldwide center of mathematics, LLC}, 2013.
	
%\bibitem{AA08} R. E. Atani and S. E. Atani, Ideal theory in commutative semirings, \textit{Bul. Acad. Stiin te Repub. Mold. Mat.}, (2008), 14--23.
	
%\bibitem{B51}S. Bourne, The Jacobson radical of a semiring, \textit{Proc. Natl. Acad. Sci.}, 37 (1951), 163--70.

%\bibitem{B52}---------, On the homomorphism theorem for semirings. \textit{Proc. Nat. Acad. Sci. U.S.A.}, 38 (1952), 118--119.

%\bibitem{B57} ---------, On multiplicative idempotents of a potent semiring, \textit{Proc. Nat. Acad. Sci.}, 42 (1956), 632--638.

%\bibitem{BZ58} --------- and H. Zassenhaus, On the semiradical of a semiring. \textit{Proc. Nat. Acad. Sci. U.S.A.}, 44 (1958), 907--914.
 
\bibitem{G99} J. S. Golan, \textit{Semirings and their applications}, Springer, 1999.

%\bibitem{G02} K. G\l azek, \textit{A guide to the literature on semirings and their applications in mathematics and information sciences, With complete bibliography}, Kluwer Academic Publishers, 2002.

%\bibitem{H15}H-C. Han, Maximal $k$-ideals and $r$-ideals in semirings, \textit{J. Algebra Appl.}, 14(10) (2015), 1250195 (13 pages).

\bibitem{H58}
M. Henriksen, Ideals in semirings with commutative addition, \textit{Amer. Math. Soc.
Notices}, 6(3) 31 (1958), 321.

%\bibitem{HW98} U. Hebisch and H. J. Weinert, \textit{Semirings: algebraic theory and applications in computer science}, World Scientific, 1998.

%\bibitem{I59} K. Ilzuka, On the Jacobson radical of a semiring, \textit{Tohuku Math. J.}, 11 (1959), 409--421. 

%\bibitem{I56} K. Is\'{e}ki, Ideal theory of semiring , \textit{Proc. Japan. Acad.}, 32 (1956), 554--559.

\bibitem{JRT22}
J. Jun, S. Ray, and J. Tolliver, Lattices, spectral spaces, and closure operators on
idempotent semirings, \textit{J. Algebra}, 594 (2022), 313--363.

%\bibitem{L95} S. LaGrassa, Semirings: ideals and polynomials, PhD thesis, University of Iowa (1995).

%\bibitem{L65} D. R. LaTorre, On $h$-ideals and $k$-ideals in hemirings, \textit{Publ. Math. Debrecen}, 12 (1965), 219--226.

%\bibitem{L15} P. Lescot, Prime and primary ideals in semirings, \textit{Osaka J. Math.}, 52(3) (2015), 721--737.

\bibitem{SA92}
M. K. Sen and M. R. Adhikari, On $k$-ideals of semirings, \textit{Internat. J. Math. \& Math. Sci.}, 15(2) (1992), 347--350.

%\bibitem{SA93} ---------, On maximal $k$-ideals of semirings, \textit{Proc. Amer. Math. Soc.} 118 (1993), 699--703.

%\bibitem{T65} D. R. LaTorre, On $h$-ideals and $k$-ideals in hemirings, \textit{Publ. Math. Debrecen}, 12 (1965), 219--226.

\bibitem{V34}
H. S. Vandiver, Note on a simple type of algebra in which the cancellation law
of addition does not hold, \textit{Bull. Am. Math. Soc.}, 40(12) (1934), 914--920.

\bibitem{V39}
H. S. Vandiver, On some simple types of semi-rings, \textit{Am. Math. Monthly}, 46(1) (1939), 22--26.

%\bibitem{WSA96} H. J. Weinert, M. K. Sen, and M. R. Adhikari, One-sided $k$-ideals and $h$-ideals in semirings, \textit{Math. Pannon.}, 7 (1996), 147--162.

\end{thebibliography}
\end{document}